\documentclass{amsart}
\usepackage{tikz-cd}
\usepackage{xcolor}
\numberwithin{equation}{section}
\usepackage{amssymb}
\usepackage[T1]{fontenc}
\usepackage[utf8]{inputenc}
\allowdisplaybreaks

\let\cal\mathcal
\def\Ascr{{\cal A}}

\def\Dscr{{\cal D}}

\def\Mscr{{\cal M}}
\def\Nscr{{\cal N}}
\def\Oscr{{\cal O}}

\def\Qscr{{\cal Q}}

\def\Sscr{{\cal S}}
\def\Tscr{{\cal T}}

\def\Vscr{{\cal V}}

\let\blb\mathbb
\def\CC{{\blb C}}

\def\QQ{{\blb Q}}

\def \AA{{\blb A}}
\def \ZZ{{\blb Z}}
\def \TT{{\blb T}}
\def \NN{{\blb N}}
\def \RR{{\blb R}}

\def\id{\text{id}}
\def\Id{\operatorname{id}}

\def\Tau{\mathcal{T}}

\def\quot{/\!\!/}

\def\Qch{\operatorname{Qch}}

\def\Spec{\operatorname {Spec}}

\def\GL{\operatorname {GL}}
\def\PGL{\operatorname {PGL}}

\def\Hom{\operatorname {Hom}}

\def\End{\operatorname {End}}
\def\RHom{\operatorname {RHom}}

\def\End{\operatorname {End}}

\def\id{{\operatorname {id}}}

\def\rk{\operatorname {rk}}

\def\r{\rightarrow}

\DeclareMathOperator{\res}{res}

\newtheorem{lemma}{Lemma}[section]
\newtheorem{proposition}[lemma]{Proposition}
\newtheorem{theorem}[lemma]{Theorem}
\newtheorem{corollary}[lemma]{Corollary}

\theoremstyle{definition}

\newtheorem{definition}[lemma]{Definition}

{
\newtheorem{claim}{Claim}

}

\theoremstyle{remark}

\newtheorem{remark}[lemma]{Remark}

\newdimen\uboxsep \uboxsep=1ex
\def\uboxn#1{\vtop to 0pt{\hrule height 0pt depth 0pt\vskip\uboxsep
\hbox to 0pt{\hss #1\hss}\vss}}

\def\uboxs#1{\vbox to 0pt{\vss\hbox to 0pt{\hss #1\hss}
\vskip\uboxsep\hrule height 0pt depth 0pt}}

\def\Vol{\operatorname{Vol}}
\def\Perf{\operatorname{Perf}}
\def\BS{\operatorname{BS}}
\usepackage{comment}

\author[Valery Lunts, \v{S}pela \v{S}penko and Michel Van den Bergh]{Valery Lunts, \v{S}pela \v{S}penko and Michel
  Van den Bergh} 
\address{Department of Mathematics, Indiana University Bloomington, Rawles Hall 251, 831 East 3rd St., Bloomington, IN 47405-7106}
\email{vlunts@indiana.edu}
\address{National Research University Higher School of Economics, Moscow}
\address[\v{S}pela \v{S}penko]{D\'epartement de Math\'ematique, Universit\'e Libre de Bruxelles, Campus de la Plaine CP 213, Bld du Triomphe, B-1050 Bruxelles}
\email{spela.spenko@ulb.be}
\address[Michel Van den Bergh]{Vakgroep Wiskunde, Universiteit Hasselt, Universitaire Campus \\
  B-3590 Diepenbeek} 
\email{michel.vandenbergh@uhasselt.be}
\address{Vrije Universiteit Brussel, Department of Mathematics and Data Science, Pleinlaan 2\\B-1050 Brussel} 
\email{michel.van.den.bergh@vub.be}
\thanks{ The first author was supported by the Basic Research Program of the National
Research University Higher School of Economics. The second author is supported by a MIS grant from the National Fund for Scientific Research (FNRS) and an ACR grant from the Université Libre de Bruxelles. The third author is a senior researcher at the Research
  Foundation Flanders (FWO). 
  This project has received funding from the European Research Council (ERC) under the European Union's Horizon 2020 research and innovation programme (grant agreement No 885203).
In addition this project has received funding from the FWO grant G0D8616N: ``Hochschild cohomology and
  deformation theory of triangulated categories''.
}

\title[Catalan numbers and noncommutative Hilbert schemes]{Catalan numbers, parking functions, permutahedra and noncommutative Hilbert schemes}

\begin{document}
\begin{abstract}
  We find an explicit $S_n$-equivariant bijection between the
  integral points in a certain zonotope in $\RR^n$, combinatorially
  equivalent to the
  permutahedron, and the set of $m$-parking functions
  of length $n$. This bijection restricts to a bijection
  between the regular $S_n$-orbits and
  $(m,n)$-Dyck paths, the number of which is given by the Fuss-Catalan
  number $A_{n}(m,1)$. Our motivation came from studying tilting
  bundles on noncommutative Hilbert schemes. As a side result we use these tilting bundles to construct a semi-orthogonal decomposition of the derived category
  of noncommutative Hilbert schemes.
\end{abstract}
\maketitle

\section{Introduction}
\subsection{Some combinatorial results}
\label{sec:combinat}
In this section we state some purely combinatorial results which give a new interpretation of parking functions and Fuss-Catalan numbers \cite{Stanleybook,StanleyCatalan} in terms of lattice points in a certain polytope related to the permutathedron. In the next section we will give the motivation behind these results.

\medskip

Let $m,n\in \NN$ and let $(e_i)_{i=1,\ldots,n}$ be the standard basis for $\ZZ^n\subset \RR^n$. We let the symmetric group $S_n$ act on $
\ZZ^n$ and $\RR^n$ by permutations.
Let $\Delta^{m,n}$ be the $S_n$-invariant zonotope which is the Minkowski sum of the intervals
\begin{align*}
  \left[0,e_i\right],\quad &1\le i\le n,\\
  \left[0,\frac{m}{2}(e_i-e_j)\right],\quad & 1\le i\neq j\le n.
\end{align*}
For $\nu:=\sum_{i=1}^n e_i$ and $\tau\in \RR$ put $\Delta^{m,n}_\tau=\Delta^{m,n}+\tau \nu$.
\begin{definition} \label{def:admissible} We say that $\tau$ is \emph{admissible} if $\tau-m(n-1)/2$ is not a rational number with denominator $\le n$. 
\end{definition}
The significance of this condition is the following:
\begin{lemma}\label{lem:admissible} (see \S\ref{sec:admissible}) $\tau$ is admissible if and only if $\partial(\Delta^{m,n}_\tau)\cap \ZZ^n=\emptyset$.
\end{lemma}
In this note we prove the following result.
\begin{proposition} \label{prop:lattice} (see \S\ref{sec:tiling})
Assume $\tau$ is admissible. Let $L=(mn+1)\ZZ^n+\ZZ\nu$. Then
\begin{equation}
  \label{eq:tiling}
    \Delta^{m,n}_\tau\cap \ZZ^n\r \ZZ^n/L:a\mapsto \bar{a}
  \end{equation}
 is an $S_n$-equivariant bijection.
\end{proposition}
This result follows very quickly from the fact that $\Delta^{m,n}_\tau$ is equivalent, in a suitable sense, to the permutahedron and hence is 
space tiling, see Proposition \ref{lem:tiling}.

\medskip

Proposition \ref{prop:lattice} allows one to relate the lattice points in $\Delta^{m,n}_\tau$ for admissible $\tau$ to \emph{parking functions}.
Recall that an $(m,n)$-parking function is a sequence of natural numbers\footnote{We assume $0\in\NN$.} $a=(a_1,\dots,a_n)\in \NN^n$
such that its
weakly increasing rearrangement $a_{i_1}\leq a_{i_2}\leq\cdots\leq a_{i_n}$ satisfies $a_{i_j}\leq m(j-1)$. Note that $S_n$ acts on parking functions by permuting indices.
Below we denote the set of $(m,n)$-parking
functions by $\Qscr^m$. According to  \cite[NOTE in \S3]{Stanley} or \cite[\S5.1]{Athanasiadis} the map
\begin{equation}
  \label{eq:parking1}
  \Qscr^m\r \ZZ^n/L:a\mapsto \bar{a}
  \end{equation}
  is an $S_n$-equivariant bijection.   Combining Proposition \ref{prop:lattice} with \eqref{eq:parking1} yields:
  \begin{corollary} If $\tau$ is admissible then there is an explicit $S_n$-equivariant bijection between lattice points in $\Delta^{m,n}_\tau$ and $(m,n)$-parking functions.
\end{corollary}
If $a=(a_1,\dots,a_n)\in \NN^n$ is weakly increasing then we say that $a$ is an $(m,n)$-\emph{Dyck path} if $a_j\le (m-1)(j-1)$. The number of $(m,n)$-Dyck
paths is
\begin{equation}
  \label{eq:dyckformula}
  A_n(m,1):=\frac{1}{mn+1}{mn+1\choose n}=\frac{1}{(m-1)n+1}{mn \choose n}
\end{equation}
and is called the $(m,n)$-\emph{Fuss-Catalan number}.
The following is clear:
\begin{lemma} There is a bijection between regular orbits of $(m,n)$-parking functions and $(m,n)$-Dyck paths which sends the orbit representative
  $a=(a_1,\ldots,a_n)\in \NN^n$ with $a_1<\ldots<a_n$ to $(a_1,a_2-1,\ldots,a_n-(n-1))$.
\end{lemma}
We thus obtain a new interpretation of the Fuss-Catalan numbers.
\begin{corollary} \label{cor:cor1}
  If $\tau$ is admissible then there is an explicit $S_n$-equivariant bijection between regular $S_n$-orbits in $\Delta^{m,n}_\tau\cap \ZZ^n$ and $(m,n)$-Dyck paths. In particular the number of such regular orbits is $A_n(m,1)$.
\end{corollary}
The claim in Corollary \ref{cor:cor1} concerning $A_n(m,1)$ was first observed by us as a consequence of the properties of the ``noncommutative Hilbert scheme''.
This is explained in \S\ref{sec:Hmn} below.

\medskip

In an appendix we will also give a second combinatorial proof of the claim about $A_n(m,1)$. The basic idea is that if $\tau$ is admissible then,
since $\Delta^{m,n}_\tau$ is a zonotope such that $\partial(\Delta^{m,n}_\tau) \cap \ZZ^n=\emptyset$, we have 
$|\Delta^{m,n}_\tau\cap \ZZ^n|=\Vol(\Delta^{m,n}_\tau)$ by Proposition \ref{prop:vol_lat_shifted} and moreover there is an explicit formula for $\Vol(\Delta^{m,n}_\tau)$
in terms of spanning trees in a suitable graph (see \S8).
Using an appropriate inclusion/exclusion argument we may upgrade this to a count of the regular orbits in $\Delta^{m,n}_\tau\cap \ZZ^n$.
\subsection{The noncommutative Hilbert scheme}
\label{sec:Hmn}
The Hilbert scheme of length $n$-sheaves on $\AA^n$
may be viewed as the moduli space of cyclic modules of dimension~$n$ over the polynomial algebra $\CC[x_1,\ldots,x_m]$. It is then natural to define
the corresponding \emph{noncommutative Hilbert scheme} $H_{m,n}$ as  the moduli space of cyclic modules of dimension $n$ over the free algebra $\CC\langle x_1,\ldots,x_m\rangle$. We recall:
\begin{proposition}[\protect{\cite{Reineke,MR1048420}}] \label{prop:strata}
  $H_{m,n}$ has a stratification consisting of affine spaces and the number of strata is given by
  the Fuss-Catalan number $A_n(m,1)$.
\end{proposition}

It is clear that $H_{m,n}$ can be described as the moduli space of stable (or equivalently semi-stable) representations with dimension vector $(1,n)$ and stability condition $(-n,1)$ \cite[Definition 1.1]{King}
of the following quiver $Q_{m,n}$:
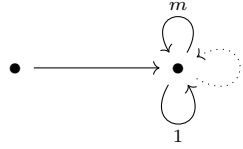
\[
  \begin{tikzcd}
    \bullet \arrow[rr] &&\bullet \arrow[out=60,in=120,loop,"m"']\arrow[out=240,in=300,loop,"1"']\arrow[out=-30,in=30,loop,dotted]
  \end{tikzcd}
\]
It follows from loc.\ cit.\ that $H_{m,n}$ can also be described as a
GIT quotient for the group
$(\CC^\ast \times \GL_n(\CC))/\{\text{center}\}\cong
\GL_n(\CC)$. More precisely we get $H_{m,n}=W^{ss,\chi}/G$ where
$G=\GL_n(\CC)$, $W=\End(\CC^n)^{\oplus m}\oplus \CC^n$ and
$W^{ss,\chi}\subset W$ is the semi-stable locus associated to the 
determinant character $\chi$.

\medskip

Using the GIT description $H_{m,n}$ we will show using \cite{HLSam,SVdB} that $H_{m,n}$ admits a family of tilting bundles. Let $\Delta^{m,n}_\tau\subset \RR^n$
be as in \S\ref{sec:combinat}. We identify $\ZZ^n$ with the character group of the diagonal torus $(\CC^\ast)^n$ in $\GL_n(\CC)$.
Let $(\ZZ^n)^+$ be the ``dominant'' part of $\ZZ^n$, i.e.\ those $(c_1,\ldots,c_n)\in \ZZ^n$ such that $c_1\ge \cdots \ge c_n$.
For $\xi\in (\ZZ^n)^+$ let $V(\xi)$ be the irreducible $\GL_n(\CC)$ representation with highest weight~$\xi$ and let $\Vscr(\xi)$ be the equivariant
vector bundle on $H_{m,n}$ corresponding to the $\GL_n(\CC)$-equivariant vector bundle $V(\xi)\otimes_k \Oscr_{W^{ss,\chi}}$ on $\Oscr_{W^{ss,\chi}}$. Put
\[
\hat{\rho}=\frac{1}{2} \sum_{i<j} (e_i-e_j)+\frac{1}{2}(n-1)\nu=(n-1,n-2,\ldots,1,0).
\]

\begin{proposition}(see \S\ref{sec:tilting}) \label{prop:tilting} 
  Let $\tau$ be admissible. 
  Then
  \begin{equation}\label{eq:tilting}
   \Tscr_\tau:= \bigoplus_{\xi \in (\ZZ^n)^+\cap (\Delta^{m,n}_{\tau}-\hat{\rho})} \Vscr(\xi)
  \end{equation}
  is a tilting bundle on $H_{m,n}$.
\end{proposition}
We refer the reader to Appendix \ref{sec:tables} for tables with tilting bundles on $H_{2,2}$, $H_{3,2}$ and $H_{4,2}$. 

Comparing the ranks of $K_0(H_{m,n})$ obtained from Propositions \ref{prop:strata} and \ref{prop:tilting} yields the identity
\begin{equation}
  \label{eq:comparing}
|  (\ZZ^n)^+\cap (\Delta^{m,n}_{\tau}-\hat{\rho})|=A_n(m,1).
  \end{equation}
  Sending $a\mapsto a+\hat{\rho}$ defines a bijection between $(\ZZ^n)^+\cap (\Delta^{m,n}_{\tau}-\hat{\rho})$ and the regular orbits in $\ZZ^n\cap \Delta^{m,n}_\tau$.
This  yields a ``geometric'' proof of the claim about $A_n(m,1)$ in Corollary \ref{cor:cor1}.

\subsection{A semi-orthogonal decomposition of the non-commutative Hilbert scheme}
We will use the tilting bundles defined in \eqref{eq:tilting} to obtain an interesting side result on the structure of $\Dscr(H_{m,n}):=D_{\Qch}(H_{m,n})$.
Recall that if $X$ is a noetherian scheme and $\Ascr$ is a sheaf of Azumaya algebras on $X$ of rank $n^2$ then the Brauer-Severi scheme $Y=\BS(\Ascr,X)$ is defined as the moduli space of
left ideals of codimension $n$ in $\Ascr$ (using the opposite convention from \cite[\S8.4]{Quillen}).
The definition of the Brauer-Severi variety was extended to singular
schemes in \cite{ArtinMumford} and in \cite{MR1048420}  (see also
\cite{MR1684876,MR1407874}) it was shown
that $H_{m,n}$ is the Brauer-Severi scheme of the so-called ``trace
ring of $m$ generic $n\times n$-matrices $\TT_{m,n}$''.

Trace rings were studied by Artin \cite{ArtinAzumaya} and Procesi \cite{Proc}, we refer the reader to \cite[\S1.4.5]{SVdB} for a short introduction. Recall that the commutative trace ring $Z_{m,n}$ of
$m$ generic $n\times n$-matrices is equal to $\Gamma(W_0)^G$ with $W_0=M_n(\CC)^{\oplus m}$ and $G=\GL_n(\CC)$, acting by conjugation. The noncommutative trace ring $\TT_{m,n}$ is the
$Z_{m,n}$-\emph{algebra of covariants} \cite[\S4.1]{SVdB} $M(M_n(\CC)):=(M_n(\CC)\otimes \Gamma(W_0))^G$. It follows from the definitions that $H_{m,n}\r \Spec Z_{m,n}$ is a standard Brauer-Severi scheme when restricted to the Azumaya locus
of $\TT_{m,n}$. This locus is non-empty and dense when $m\ge 2$.

If $\xi\in \ZZ^n$ is a weight for $G$ then we define its \emph{color} as $c(\xi)=\sum_i \xi_i$; i.e.\ it is the weight of $\xi$ when restricted to the center of $G$.
For $c\in \ZZ$ we put
\begin{equation}
  \label{eq:Vc}
      V_\tau(c):= \bigoplus_{\xi \in (\ZZ^n)^+\cap (\Delta^{m,n}_{\tau}-\hat{\rho}),c(\xi)=c} V(\xi)
  \end{equation}
  \begin{proposition} \label{prop:sod}
    Assume that $\tau$ is admissible and $m\ge 2$. There is a semi-orthogonal decomposition (depending on $\tau$) 
  \begin{equation}
    \label{eq:sod}
    \Dscr(H_{m,n})=\langle D(R_{\tau,u}),\ldots, D(R_{\tau,u+n-1}) \rangle
  \end{equation}
  linear over $Z_{m,n}$ with $u=\lceil n\tau\rceil-n(n-1)/2$ such that
  \[
    R_{\tau,c}=M(\End(V_\tau(c))).
  \]
  The restriction of $R_{\tau,c}$  to the Azumaya locus of $\TT_{m,n}$ is Morita equivalent to the restriction of $\TT_{m,n}^{\otimes c}$.
\end{proposition}
The proof is based on partitioning the tilting bundle $\Tscr_\tau$ into semi-orthogonal parts.
\begin{remark} The reader will guess that the restriction of \eqref{eq:sod} to the Azumaya locus of $\TT_{m,n}$ 
 is a rotation of the usual semi-orthogonal decomposition of a Brauer-Severi scheme (see \cite[Theorem 5.1]{MR2571702}). One may show that this guess is correct.
\end{remark}
\begin{remark} From \eqref{eq:sod} we get a $\tau$-dependent decomposition
  \[
    A_n(m,1)=\rk K_0(H_{m,n})=\sum_{c=u}^{u+n-1}  \rk K_0(R_{\tau,c}).
  \]
  It would be interesting to see if this decomposition can be made concrete using some of the many combinatorial interpretations of $A_n(m,1)$.
 \end{remark}

 \subsubsection{Related work}
In \cite{TudorToda} the authors construct a semi-orthogonal decomposition of $\Dscr(H_{3,n})$. Their decomposition is more refined as it also involves categories  
generated by suitable $\Vscr(\xi)$ on approprate smaller noncommutative Hilbert schemes.  Moreover, they construct a semi-orthogonal decomposition of the category of matrix factorisations on 
$H_{3,n}$ with a super-potential whose critical locus is the Hilbert scheme of points.

\section{Acknowledgement}
We thank Cesar Ceballos for very helpful discussions, in particular
for teaching us about parking functions which lead to the simple proof of Corollary \ref{cor:cor1}.
We further thank him for
making us aware of a plethora of incarnations and variations of Catalan numbers and for providing references.

While this work was in its final stages, the second author attended  Tudor P${\rm\breve{a}}$durariu's interesting lecture in the workshop ``Representation theory and flag or quiver varieties'' (Paris, June 2022) during which he reported on joint work with Yokinobu Toda in which, among others, a semi-orthogonal decomposition of $\Dscr(H_{3,n})$ is constructed.
We thank him for sharing this work with us \cite{TudorToda}.

\section{Zonotopes}
\subsection{Generalities}
\label{sec:generalities}
A zonotope $Z$ in a finite dimensional $\RR$-vector space $V$ is a subset of the form $t+\sum_{i=1}^u [\beta_i,\gamma_i]$. The
vectors $\gamma_i-\beta_i$ are the \emph{defining vectors} of the zonotope. The faces of a zonotope are easy to compute (see e.g. \cite[\S2]{McMullen}). They are of the
form
\begin{equation}
\label{eq:faces}
t+\sum_{\langle\lambda,\gamma_i-\beta_i\rangle>0}\gamma_i+\sum_{\langle\lambda,\gamma_i-\beta_i\rangle<0}\beta_i+
\sum_{\langle \lambda,\gamma_i-\beta_i\rangle=0} [\beta_i,\gamma_i]
\end{equation}
for $\lambda\in V^\ast$. In particular if $Z$ is full dimensional then the facets correspond to those~$\lambda$ such that the defining vectors in the kernel of $\lambda$ span a hyperplane in~$V$.
Note that the facets come in parallel pairs, and the companion to a facet corresponding to $\lambda$ is the one corresponding to $-\lambda$.
\begin{remark}
\label{rem:edges}
The defining vectors of a zonotope are of course not unique. However if we restrict to the
case that no two defining vectors are parallel then the defining vectors are unique - up to sign.
Indeed it follows from \eqref{eq:faces} that they are given by the edges of the zonotope.
We may always reduce to the case that no two defining vectors are parallel 
by taking the sum of all groups of parallel intervals. 
\end{remark}
For use below we define the \emph{tiling lattice} as the subgroup of $V$ spanned by the vectors
\begin{equation}
\label{eq:tiling_lattice}
t_\lambda=\sum_{\langle\lambda,\gamma_i-\beta_i\rangle>0}(\gamma_i-\beta_i)-\sum_{\langle\lambda,\gamma_i-\beta_i\rangle<0}(\gamma_i-\beta_i)
\end{equation}
where $\lambda$ runs through the $\lambda\in V^\ast$ defining facets (as explained above).
\subsection{Space tiling zonotopes}
Let $Z\subset V$ be a full dimensional zonotope where $\dim V=n$.
We say that $Z$ is \emph{space tiling} if there is a lattice $L\subset V$
such that $V=\bigcup_{l\in L}(l+Z)$ and such that 
$Z\cap (l+Z)$ for $l\in L$ is a face in both $Z$ and $l+Z$. It is easy to see that
$L$ must be equal to the tiling lattice of $Z$.

The following result gives an easy way of recognizing space tiling zonotopes.
\begin{proposition}[{\cite[\S2.II]{McMullen}}] \label{prop:tiling}
A full dimensional zonotope is space
  tiling if and only if every $n-2$-dimensional subspace of $V$
  spanned by the defining vectors is contained in $2$
  or $3$ (and not more) hyperplanes spanned by the defining vectors.
\end{proposition}

\begin{lemma} \label{cor:tiling}
  Assume that $M\subset V$ is a lattice and that $Z$ is a translation of a zonotope with vertices in  $M$ such that in addition $M\cap \partial Z=\emptyset$.
Assume furthermore that~$Z$ is space tiling with tiling lattice $L$ (in particular $L\subset M$). 
  Then the map
  \[
    Z\cap M\mapsto M/L:m\mapsto \bar{m}
    \]
  is a bijection.
\end{lemma}
\begin{proof} By the definition of space tiling we have an $L$-equivariant decompostion:
  \[
    M=\coprod_{l\in L} (l+M\cap Z).
  \]
  Quotienting out the $L$-action on both sides yields the lemma.
\end{proof}
\section{The structure of ${\Delta^{m,n}_\tau}$}\label{sec:faces}
 One may apply \eqref{eq:faces} to  determine the facets of
 ${\Delta^{m,n}_\tau}$.
 One verifies using \eqref{eq:faces} that the facets of $\Delta^{m,n}_\tau$ are defined by $\pm w\lambda_k$ for $w\in S_n$,
 for $k=1,\ldots,n$, $\lambda_k=(1^k,0^{n-k})$. More precisely the facets associated to $\pm w\lambda_k$ are respectively of the form
 \begin{equation}
   \label{eq:facets}
\begin{aligned}
F_{w,k}^+&=\tau\nu+(1/2)m\sum_{ (w\lambda_k)_i=1,(w\lambda_k)_j=0}(e_i-e_j)+\sum_{(w\lambda_k)_i=1}e_i\\
&\qquad\qquad\qquad +
(1/2)m\sum_{(w\lambda_k)_i=(w\lambda_k)_j}[0,e_i-e_j]+\sum_{(w\lambda_k)_i=0}[0,e_i],\\
F_{w,k}^-&=\tau\nu+(1/2)m\sum_{ (w\lambda_k)_i=0,(w\lambda_k)_j=1}(e_i-e_j)\\
&\qquad\qquad\qquad +
(1/2)m\sum_{(w\lambda_k)_i=(w\lambda_k)_j}[0,e_i-e_j]+\sum_{ (w\lambda_k)_i=0}[0,e_i].
\end{aligned}
\end{equation}

\begin{proposition}\label{lem:tiling}
The zonotope $\Delta^{m,n}_\tau$ is space tiling with tiling lattice  $(mn+1)\ZZ^n+\ZZ(1,\dots,1)$.
\end{proposition}
\begin{proof} 
We identify $\RR^n$ with the hyperplane in
$\RR^{n+1}$ given by $x_0+\cdots+x_n=0$. Then ${\Delta^{m,n}_\tau}$ is a translation of the zonotope with the defining vectors $e_i-e_0$, $m(e_j-e_i)$, $1\leq i<j\leq n$. On the other hand, the permutahedron is the Minkowski sum of $[e_i,e_j]$, $0\leq i<j\leq n$; i.e. it is a zonotope with the defining vectors $e_j-e_i$,  $0\leq i<j\leq n$ (see e.g. \cite[Theorem 9.4]{BeckRobins}). It is known that the permutahedron is space tiling, see e.g. \cite[Ex. 9.12]{BeckRobins}. By Proposition \ref{prop:tiling}, the space tiling property does not depend on the length of the defining vectors. Thus, ${\Delta^{m,n}_\tau}$ is also space tiling.

Now we compute the tiling lattice. As indicated above the facets of $\Delta^{m,n}_k$ are determined by $\pm w\lambda_k$.
Furthermore in the formula \eqref{eq:tiling_lattice} we have $t_{-\lambda}=-t_\lambda$ so it suffices to compute the lattice spanned by $(t_{w\lambda_k})_{w,k}$
Let $S=\{i\mid (w\lambda_k)_i=1\}$ and hence $S^c=\{1,\dots,n\}\setminus S=\{i\mid (w\lambda_k)_i=0\}$. Note that $|S|=k$, $|S^c|=n-k$.
We write $\delta_{iS}$ for the characteristic function of $S$; i.e.\ $\delta_{iS}=1$ if $i\in S$ and $\delta_{iS}=0$ otherwise. With $(\delta_{iS})_{1\leq i\leq n}$ we denote $(\delta_{1S},\dots, \delta_{nS})$. Then
\begin{align*}
t_{w\lambda_k}&=(m/2)\sum_{i\in S,j\in S^c}(e_i-e_j)-(m/2)\sum_{i\in S^c,j\in S}(e_i-e_j)+\sum_{i\in S}e_i\\
&=m\sum_{i\in S,j\in S^c}(e_i-e_j)+\sum_{i\in S}e_i\\
&=(m(n-k)+1)(\delta_{iS})_{1\leq i\leq n}-mk(\delta_{iS^c})_{1\leq i\leq n}\\
&=
(mn+1)(\delta_{iS})_{1\leq i\leq n}-mk(1,\dots,1). 
\end{align*}
Setting $k=n$ we obtain $t_{w\lambda_n}=(1,\dots,1)$. Adding a  suitable multiple of $t_{w\lambda_n}$ to $t_{w\lambda_k}$ we then obtain that the lattice is generated by $(1,\ldots,1)$, 
$(mn+1)(\delta_{iS})_{1\leq i\leq n}$ for  $S\subset \{1,\dots,n\}$, which easily implies our claim.
\end{proof}

\subsection{Proof of Lemma \ref{lem:admissible}}
\label{sec:admissible}
Fix $1\le k\le n$. We  have to understand when $F_{w,k}^{\pm}$ contains an
integral point for $w\in S_n$. Renumbering the $(e_i)_i$ we may assume that $w=\id$ where
$\id$ is the identity in $S_n$. Moreover it follows from
\eqref{eq:facets} that  $F^-_{\id,k}$ contains a rational point if and only if
$F^+_{\id,k}$ does too. Hence we are reduced to understanding to when $F_k:=F_{\Id,k}^{+}$ contains an integral point.

\medskip

Let $L_k=\sum_{i>k}\ZZ e_i + \sum_{i,j\leq k\text{ or }i,j>k} \ZZ(e_i-e_j)$. Then $(L_k)_\RR$ is the hyperplane through the origin which is parallel to $F_k$.
Sending
$e_i$ to $1$ for $i\le k$ and to $0$ for $i>k$ identifies $\ZZ^n/L_k$ with $\ZZ$ and $\RR^n/(L_k)_\RR$ with $\RR$.
We regard $\ZZ^n/L_k$ as being contained in $\RR^n/(L_k)_\RR$.

Let $\overline{F_k}$ be the image of $F_k$ in $\RR^n/(L_k)_\RR$. Note that $\overline{F_k}$ is a singleton.
\begin{claim}\label{claim1} $F_k$ contains an integral point
if and only if
$\overline{F_k}\subset \ZZ^n/L_k$.
\end{claim}
One direction is obvious.
To see the other direction assume $x\in \ZZ^n$ is such that $(x+(L_k)_\RR)\cap F_k\neq \emptyset$. Then $-x+F_k$ is
a translation of 
\begin{equation}
  \label{eq:translated}
 (1/2)m\sum_{i,j\leq k\text{ or }i,j>k}[0,e_i-e_j]+\sum_{i>k}[0,e_i]
\end{equation}
inside $(L_k)_\RR$. Now \eqref{eq:translated} is itself a translation of a full dimensional lattice polytope and any translation of a full dimensional 
lattice polytope contains an integral point (e.g.\ by Proposition \ref{prop:vol_lat_shifted}). Hence there exists some $y\in \ZZ^n$ such that
$y\in -x+F_k$. We conclude $x+y\in F_k\cap \ZZ^n$.
\begin{claim}
  $\overline{F_k}$ is contained in $\ZZ^n/L_k$ if and only if $\tau-m(n-1)/2\in (1/k)\ZZ$,
  \end{claim}
  To prove this we use the identifications  $\ZZ^n/L_k\cong\ZZ$,  $\RR^n/(L_k)_\RR\cong\RR$ given above. By \eqref{eq:facets} $\overline{F_k}$
is the singleton $\{k\tau+mk(n-k)/2+k\}$. This is contained in $\ZZ$ 
if and only if $\tau+m(n-k)/2\in (1/k) \ZZ$ which may be rewritten as $\tau-m(n-1)/2\in -m(k-1)/2+(1/k)\ZZ=(1/k)\ZZ$.

\medskip

Combining Claims 1 and 2 finishes the proof.
\subsection{Proof of Proposition \ref{prop:lattice}}
\label{sec:tiling}
The fact that \eqref{eq:tiling} is a bijection follows directly from Lemma \ref{cor:tiling} combined with Proposition \ref{lem:tiling}. The fact that \eqref{eq:tiling}
is $S_n$-equivariant is clear.
\section{Proof of Proposition \ref{prop:tilting}}
\label{sec:tilting}
Let the setting be as in \S\ref{sec:Hmn} in the introduction. Our aim is
to apply \cite{HLSam} but the representation $W$ is not
``quasi-symmetric'' (see \cite{SVdB}). So the theory of \cite{HLSam}
does not apply on the nose.  We circumvent this by relating $H_{m,n}$
to the moduli space $\tilde{H}_{m,n}$ of stable (or equivalently
semi-stable) representations with dimension vector $(1,n)$ and
stability condition $(-n,1)$ of the quiver $\tilde{Q}_{m,n}$:
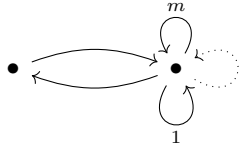
\[
  \begin{tikzcd}
    \bullet \arrow[rr,out=20,in=160] &&\bullet \arrow[out=60,in=120,loop,"m"']\arrow[out=240,in=300,loop,"1"']\arrow[out=-30,in=30,loop,dotted]
    \arrow[ll,out=200,in=-20]
  \end{tikzcd}
\]
This quiver is symmetric so the corresponding GIT setting is symmetric. Moreover it is easy to see that a representation
of $\tilde{Q}_{m,n}$ is semi-stable if and only if its restriction ot $Q_{m,n}$ is semi-stable.
It then follows by 
a descent argument that the map $\pi:\tilde{H}_{m,n}\r H_{m,n}$, obtained by forgetting the left pointing arrow, is a vector bundle.

We now assume that~$\tau$ is admissible. By Lemma \ref{lem:admissible}, $\partial(\Delta^{m,n}_\tau)\cap \ZZ^n=\emptyset$.  For $\xi\in \ZZ^n$ let $\tilde{\Tscr}_\tau$ be defined as $\Tscr_\tau$ but using the quiver
$\tilde{Q}_{m,n}$. Then~$\tilde{\Tscr}_\tau$ is a tilting bundle on $\tilde{H}_{m,n}$  by \cite[Corollary 4.2]{HLSam}.
\footnote{The polytope used in \cite{HLSam} to construct tilting objects is denoted  by $(1/2)\bar{\Sigma}-\rho+\delta$ in loc.\ cit., where~$\rho$ is half the sum of the positive roots and (in our setting) $\delta=\tau' \nu$ for suitable~$\tau'$. See \cite[Lemma 2.9, Corollary 4.2]{HLSam}.
  Using the definition of $\Sigma$ in \cite[\S2]{HLSam} one checks that $\Delta^{m,n}_\tau-\hat{\rho}=(1/2)\bar{\Sigma}-\rho+\tau'\nu$ where $\tau=\tau'-1+n/2$, so that we may indeed use
  the results of \cite{HLSam}. Note that the exact relation between $\tau$ and $\tau'$ is not important here.}
\begin{lemma} Let $\pi:\tilde{X}\r X$ be a morphism of quasi-compact quasi-separated schemes with a section $i:X\r \tilde{X}$. Let $\Tscr\in \Perf(X)$  such that
  $L\pi^\ast \Tscr$ is a tilting object in $D_{\Qch}(\tilde{X})$. Then $\Tscr$ is itself tilting.
\end{lemma}
\begin{proof}
  We first need to prove that $\Tscr$ is a generator of  $D_{\Qch}(X)$, i.e.\ that $\Tscr^\perp=0$. Let $\Nscr\in D_{\Qch}(X)$ be such
  that $\RHom_X(\Tscr,\Nscr)=0$. We compute
  \begin{align*}
    \RHom_X(\Tscr,\Nscr)&=\RHom_X(\Tscr,R\pi_\ast Ri_\ast \Nscr)\\
    &=\RHom_{\tilde{X}}(L\pi^\ast \Tscr,Ri_\ast \Nscr).
  \end{align*}
 Since $(L\pi^\ast \Tscr)^\perp=0$ deduce $Ri_\ast \Nscr=0$ and hence $\Nscr=R\pi_\ast Ri_\ast\Nscr=0$. So $\Tscr^\perp=0$.
  We also have
  \begin{align*}
    \RHom_{\tilde{X}}(L\pi^\ast \Tscr,L\pi^\ast \Tscr)&= \RHom_{X}(\Tscr,R\pi_\ast L\pi^\ast \Tscr).
  \end{align*}
  Since $\tilde{X}\r X$ is split, the unit map $\Tscr\r R\pi_\ast L\pi^\ast \Tscr$ is also split. It follows that
  $  \RHom_{X}(\Tscr,\Tscr)$ is a direct summand of  $\RHom_{\tilde{X}}(L\pi^\ast \Tscr,L\pi^\ast \Tscr)$. Hence if $L\pi^\ast \Tscr$ has no non-zero self-extensions
  then neither has $\Tscr$.
\end{proof}

\section{A semi-orthogonal decomposition for the noncommutative Hilbert scheme}

We use the same notations as in \S\ref{sec:tilting}. Let $\tilde{W}$
be the GIT setting corresponding to $\tilde{Q}_{m,n}$,
i.e. $(\tilde{W},G)$ with
$\tilde{W}=\CC^n\oplus (\CC^n)^\ast\oplus M_n(\CC)^{\oplus m}$.  We
will consider the $\CC^\ast$-action on $\tilde{W}$ obtained by scaling
the right pointing arrow in $\tilde{Q}_{m,n}$. This action 
commutes with the $G$-action and we have $\tilde{W}\quot \CC^\ast=W$,
$\tilde{H}_{m,n}\quot \CC^\ast=H_{m,n}$.

 Below we will define some notions for the GIT setting $(W,G)$. Similar notion
 related to the GIT setting $(\tilde{W},G)$ will be decorated with a tilde. 
 
\medskip

Recall that $H_{m,n}=W^{ss,\chi}/G$.
For $\tau \in \RR$ let $\Mscr_\tau\subset D_{\Qch}(W/G)$ be the smallest cocomplete subcategory of $D_{\Qch}(W/G)$ containing the $G$-equivariant $\Oscr_W$-modules
$V(\xi)\otimes \Oscr_W$ for $\xi\in (\Delta^{m,n}_\tau-\hat{\rho})\cap (\ZZ^n)^+$.
\begin{lemma}\label{lem:restriction} Assume that $\tau$ is admissible. The restriction map
  \[
    \res :D_{\Qch}(W/G)\r D_{\Qch}(W^{ss,\chi}/G)
  \]
  restricts to an equivalence $\Mscr_\tau \cong  D_{\Qch}(W^{ss,\chi}/G)$.
\end{lemma}
\begin{proof}
  Using the fact that $\Tscr_\tau$ (cfr Proposition \ref{prop:tilting}) is a tilting bundle it  suffices to prove that when $\xi,\xi'\in (\ZZ^n)^+\cap (\Delta^{m,n}_{\tau}-\hat{\rho})$ the restriction map defines an isomorphism
  \[
\Hom_W(V(\xi)\otimes \Oscr_W,V(\xi')\otimes \Oscr_W)^G \r \Hom_{H_{m,n}}(\Vscr(\xi),\Vscr(\xi')).
\]
This follows by applying $(-)^{\CC^\ast}$ to
\[
\Hom_{\tilde{W}}(V(\xi)\otimes \Oscr_{\tilde{W}}, V(\xi')\otimes \Oscr_{\tilde{W}})^G \r \Hom_{\tilde{H}_{m,n}}(\tilde{\Vscr}(\xi),\tilde{\Vscr}(\xi'))
\]
together with \cite[Theorem 3.2]{HLSam}. 
\end{proof}
Put
\[
      \Tscr_{\tau,c}:= \bigoplus_{\xi \in (\ZZ^n)^+\cap (\Delta^{m,n}_{\tau}-\hat{\rho}),c(\xi)=c} \Vscr(\xi)
\]
\begin{lemma}
  \label{lem:homs}
    Assume $\tau$ that is admissible.
    \begin{enumerate}
    \item \label{item:1}We have $\Tscr_{\tau,c}= 0$ unless $c\in \{u,\ldots, u+n-1\}$ for $u=\lceil n\tau\rceil-n(n-1)/2$. 
    \item Assume that $m\geq 2$. Then $\Tscr_{\tau,c}\neq  0$ for $c\in \{u,\ldots, u+n-1\}$ for $u=\lceil n\tau\rceil-n(n-1)/2$.
        \item We have $\Hom_{H_{m,n}}(\Tscr_{\tau,c},\Tscr_{\tau,c'})=0$ if $c>c'$.
    \item We have
     $
        \Hom_{H_{m,n}}(\Tscr_{\tau,c},\Tscr_{\tau,c})=(\End(V_\tau(c))\otimes \Gamma(W_0))^{\bar{G}}
        $ where $\bar{G}=G/Z(G)=\PGL_n(\CC)$ and where $V_\tau(c)$ was defined in \eqref{eq:Vc}.
\end{enumerate}
\end{lemma}
\begin{proof}
  \begin{enumerate}
  \item 
  If $\xi\in \Delta^{m,n}_{\tau}-\hat{\rho}$ then it follows from the definition of $\Delta^{m,n}_\tau$ that $c(\xi)=\sum_i\xi_i=v+n\tau-c(\hat{\rho})=v+n\tau-n(n-1)/2$ for $v\in \{0,\dots,n\}$.
By admissibility $n\tau$ is non-integral. It follows that $c(\xi)$ can only be an integer when it is as in the statement of the lemma.

\item 
  It suffices (by \eqref{item:1}) to find $n$ points $(\xi^i)_{i=1}^n$
  in $(\ZZ^n)^+\cap(\Delta_{\tau}^{m,n}-\hat{\rho})$ whose $c$-values
  give all congruence classes modulo $n$.  
We will construct such $(\xi^i)_i$ such that each $\xi=\xi^i$ satisfies the following
additional condition:
\begin{enumerate}
\item[(I)] $\xi$
  does not belong to any facet of $\Delta^{m,n}_{\tau'}-\hat{\rho}$
  for any $\tau'\in \RR$ except for possibly the ``extreme'' facets
$
  F_{\id,n}^+-\hat{\rho}, F_{\id,n}^--\hat{\rho},
$
  cf. \S\ref{sec:faces}. 
\end{enumerate}
Assume that for a particular $\tau$ we have constructed $(\xi^i)_i$, covering $n$ congruence classes, such
that in addition (I) is satisfied for each $\xi=\xi^i$.

Then changing $\tau$ by crossing an
  non-admissible value $\tau_0$ to $\tau'$, we either still have
  $\xi\in \Delta_{\tau'}^{m,n}-\hat{\rho}$ or
  $\xi\not \in \Delta_{\tau'}^{m,n}-\hat{\rho}$ and $\xi$ lies on one
  of the two extreme facets of $\Delta_{\tau_0}^{m,n}-\hat{\rho}$. In this
  case $\xi+\sum_i e_i$ or $\xi-\sum_i e_i$ lies in
  $\Delta_{\tau_0}^{m,n}-\hat{\rho}$ on the opposite facet and hence
  belongs to $(\ZZ^n)^+\cap(\Delta_{\tau'}^{m,n}-\hat{\rho})$ and its
  $c$-value is congruent to $c(\xi)$ modulo $n$. Hence we keep $n$
  points in $(\ZZ^n)^+\cap(\Delta_{\tau'}^{m,n}-\hat{\rho})$ whose
  $c$-values give all congruence classes modulo $n$.
Thus it is sufficient to construct the $(\xi^i)_i$ satisfying (I) and covering
$n$ congruence classes under the assumption that $0< \tau<1/n$.

For $0< \tau<1/n$ and $0\leq l\leq n-1$ let
\[
\xi=
(n-1,n-3,\dots,n-2l+1,n-2l,n-2l-2,\dots,-n+2)-\hat{\rho}.
\]
One verifies that $\xi \in (\ZZ^n)^+$. 
We claim that $\xi \in \Delta_{\tau}^{m,n}-\hat{\rho}$ and that $\xi$ satisfies~(I).
Since  $c(\xi)=n-l-c(\hat{\rho})$, the possible $\xi$ cover $n$ congruence classes and hence we are done if we can prove this claim.

We will verify both parts of the claim simultaneously. To show that $\xi$ satisfies condition (I)  we must understand the facets of $\Delta_{\tau'}^{m,n}$ for all $\tau'\in \RR$.
So as before let $F_{w,k}^\pm$ be the facets of $\Delta_{\tau'}^{m,n}$ and let $f^\pm\in F_{w,k}^\pm$. By \eqref{eq:facets} we obtain:
\begin{equation}
\label{eq:listfaces}
\begin{aligned}
\langle w\lambda_k,f^+\rangle&=\tau'k+(1/2)mk(n-k)+k,\\
\langle -w\lambda_k,f^-\rangle&=-\tau'k+(1/2)mk(n-k).
\end{aligned} 
\end{equation}

These equations define the supporting hyperplanes of $\Delta_{\tau'}^{m,n}$.
On the other hand, direct verification using the definition of $\xi$ yields:
\begingroup
\begin{align}\label{eq:oth}
\langle w\lambda_k,\xi+\hat{\rho}\rangle& \leq 
\begin{cases}
k(n-k)& \text{if $k\leq l$},\\\nonumber
k(n-k)+k-l&\text{if $k>l$}, 
\end{cases}\\
&=k(n-k)+\max\{0,k-l\}\\
\langle -w\lambda_k, \xi+\hat{\rho}\rangle&\leq 
\begin{cases}
k(n-k)-k& \text{if $k\leq n-l$},\\\nonumber
k(n-k)-(n-l)&\text{if $k>n-l$}
\end{cases}\\\nonumber
&=k(n-k)-\min\{k,n-l\}
\end{align}
\endgroup
with equalities only possible if $w=\id$. 
Using our hypotheses $m\ge 2$ and $0<\tau<1/n$, comparing \eqref{eq:oth} with \eqref{eq:listfaces} for $\tau'=\tau$ we obtain  $\xi+\hat{\rho}\in \Delta_{\tau}^{m,n}$. It remains to show condition (I), i.e.\ that $\xi+\hat{\rho}\not\in F_{w,k}^\pm$ for $(w,k)\neq (\id,n)$. By applying $c(-)$ to $\xi+\hat{\rho}$, $\Delta_{\tau'}^{m,n}$ 
we note  that if $\xi+\hat{\rho}\in \Delta_{\tau'}^{m,n}$ then $n\tau'\leq n-l\leq n\tau'+n$, or $-l/n\le \tau'\le (n-l)/n$. If this holds then we get from \eqref{eq:listfaces} and \eqref{eq:oth}, together with the assumption $m\ge 2$:
\begin{align*}
\langle w\lambda_k,f^+\rangle& \geq -\frac{l}{n}k+(1/2)mk(n-k)+k\\&\geq k(n-k)+{\rm max}\{0,k-l\}\geq\langle w\lambda_k,\xi+\hat{\rho}\rangle,\\
\langle -w\lambda_k,f^-\rangle& \geq -\frac{n-l}{n}k+(1/2)mk(n-k)\\&\geq k(n-k)-{\rm min}\{k,n-l\}\geq\langle -w\lambda_k,\xi+\hat{\rho}\rangle.
\end{align*}
If one of these chained inequalities is actually an equality  (which would be the case if $\xi+\hat{\rho}\in F^+_{w,k}\cup F^-_{w,k}$)
then it follows that $(w,k)=(\id,n)$. Indeed in both equations the middle inequality
can only be an equality if $k=n$, while the last inequality
can only be an  equality if $w=\id$, as mentioned after \eqref{eq:oth}.
Hence the claim follows.


\item We need to prove that $\Hom_{H_{m,n}}(\Vscr(\xi),\Vscr(\xi'))=0$ when  $\xi,\xi'\in (\ZZ^n)^+\cap (\Delta^{m,n}_{\tau}-\hat{\rho})$ and $c(\xi)>c(\xi')$. By  Lemma \ref{lem:restriction}
  we have
  \begin{align*}
    \Hom_{H_{m,n}}(\Vscr(\xi),\Vscr(\xi'))&=\Hom_{W}(V(\xi)\otimes \Oscr_{W}, V(\xi')\otimes \Oscr_{W})^G\\
                                         &=(\Hom(V(\xi),V(\xi'))\otimes \Gamma(W))^G\\
                                         &\subset (\Hom(V(\xi),V(\xi'))\otimes \Gamma(W))^{Z(G)}\\
    &=0
  \end{align*}
  where in the fourth line we use that fact that the weights of $\Gamma(W)$ with respect to $Z(G)=\CC^\ast$ are $\le 0$ and by hypothesis $c(\xi)>c(\xi')$.
\item This is proved by a similar computation. Assume $c(\xi)=c(\xi')$
  \begin{align*}
       \Hom_{H_{m,n}}(\Vscr(\xi),\Vscr(\xi'))&=\Hom_{W}(V(\xi)\otimes \Oscr_{W}, V(\xi')\otimes \Oscr_{W})^G\\
                                         &=(\Hom(V(\xi),V(\xi'))\otimes \Gamma(W))^G\\
                                         &=(\Hom(V(\xi),V(\xi'))\otimes \Gamma(W)^{Z(G)})^{\bar{G}}\\
                                         &=(\Hom(V(\xi),V(\xi'))\otimes \Gamma(W_0))^{\bar{G}}.\qedhere
    \end{align*}
 \end{enumerate}
\end{proof}
Recall that if a reductive group acts on a variety $X$ then the stable locus $X^s\subset X$ is the set of points with closed orbit and finite stabilizer.
\begin{lemma} \label{lem:inverse} The inverse image in $W_0$ of the Azumaya locus $U$ of~$\TT_{m,n}$ in $\Spec Z_{m,n}$ is equal to the $\bar{G}$-stable locus $W^s_0$ in $W_0$. The stabilizer of
  every point in $W_0^s$ is trivial. 
\end{lemma}
\begin{proof}
  By Artin's Theorem \cite[Theorem (8.3)]{ArtinAzumaya} the Azumaya locus of $\Spec Z_{m,n}$ corresponds to the simple $m$-dimensional representations of the $n$-loop quiver. It is well-known
  that this is precisely the stable locus \cite{MR958897}. If $V$ is a representation of a quiver $Q$, considered as a point in the representation space of $Q$,
  then its stabilizer  is connected (see \cite[Propositon 2.2.1]{Brion2}). 
  It follows that a stable representation has trivial stabilizer.
\end{proof}
\begin{remark} Note that Lemma \ref{lem:inverse} and the results below that depend on it, only have content for $m\ge 2$ since for $m=0,1$ we have $U=\emptyset$.
\end{remark}
Let $c\in \ZZ$. We say that a representation of $G$ has color $c$ if $Z(G)$ acts with character $c$. Note that if $\xi\in (\ZZ^n)^+$ has color $c$ then so does $V(\xi)$.
\begin{lemma} Let $V$, $V'$ be two non-zero $G$-representations with the same color. Then the restrictions to $U$ of $(\End(V)\otimes \Gamma(W_0))^{\bar{G}}$ and $(\End(V')\otimes \Gamma(W_0))^{\bar{G}}$ are Morita equivalent.
\end{lemma}
\begin{proof} There is a $\bar{G}$-equivariant nondegenerate Morita context between $\End(V)\otimes \Gamma(W_0)$ and $\End(V')\otimes \Gamma(W_0)$ given by $\Hom(V,V')\otimes \Gamma(W_0)$ and $\Hom(V',V)\otimes \Gamma(W_0)$. When restricted to $W_0^s$ this descends to a non-degenerate Morita context between the restrictions of  $(\End(V)\otimes \Gamma(W_0))^{\bar{G}}$ and $(\End(V')\otimes \Gamma(W_0))^{\bar{G}}$.
\end{proof}
\begin{corollary} \label{cor:restriction} Assume that $V$ is a $G$-representation with color $c$. Then the restriction of $(\End(V)\otimes \Gamma(W_0))^{\bar{G}}$ to $U$ is Morita equivalent to the restriction of
  $\TT_{m,n}^{\otimes c}$ to $U$.
\end{corollary}
\begin{proof} This follows from the fact that $(\CC^n)^{\otimes c}$ has color $c$.
\end{proof}
\begin{proof}[Proof of Proposition \ref{prop:sod}]
The proof follows by combining Proposition \ref{prop:tiling}, Lemma \ref{lem:homs} and Corollary \ref{cor:restriction}.   
  \end{proof}
\appendix
\section{A second proof of the numerical claim in  Corollary~\ref{cor:cor1}}
We recall some results on the combinatorics of zonotopes.
\subsection{Volumes of zonotopes and lattice points}
We first recall that any zonotope can be tiled by elementary (cubical) zonotopes.

\begin{proposition}\cite[\S5]{Shephard} \label{prop:mcmullen}
Assume $Z$ is a full dimensional zonotope in $V$ of the form
  $t+\sum_{i\in H} [0,\beta_i]$ with $t\in V$ and $(\beta_i)_i\in V$. Then $Z$ can be tiled
by elementary zonotopes of the form $t+\sum_{i\in S} \beta_i+\sum_{j\in B}[0,\beta_j]$
for $S,B\subset H$, with every $B$ such that
$(\beta_i)_{i\in B}$ is a basis of $V$ occurring exactly once and such that moreover every facet 
of an elementary zonotope that appears in the tiling is contained in the translation of a facet of $Z$ by a composition of translations
by $\pm\beta_k$ for $k\in H$.
\end{proposition}
From this we can deduce a formula for the number of lattice points of $Z$ in case there are no lattice points on the boundary.
\begin{proposition}\label{prop:vol_lat_shifted} Let $M\subset V$ be a lattice and let $Z\subset V$ be a translation of a full dimensional
zonotope with vertices in $M$. Then
\[
|Z\cap M|\ge \Vol_M(Z)
\]
with equality if $\partial Z\cap M=\emptyset$.
\end{proposition}

\begin{proof} 
By shifting $Z$ slightly (which does not increase $|Z\cap M|$) we may ensure that~$M$ does not intersect any facet of the tiling elementary zonotopes.
Then  Proposition \ref{prop:mcmullen} reduces us to the case that $Z$ is an elementary
zonotope which is trivial. 
\end{proof}
Combining this with Stanley's formula \cite[Theorem 2.2]{Stanley1} for $\Vol_M(Z)$ we  obtain a formula for $|Z\cap M|$ in case $\partial Z\cap M=\emptyset$.
\begin{corollary}\label{thm:stanley}\cite[Theorem 2.2]{Stanley1}
Let the setting be as in Proposition \ref{prop:vol_lat_shifted}. 
Then $|Z\cap M|=\sum_S h(S)$ where $S$ ranges over all maximal linearly independent subsets of $\{\beta_i\mid i\in H\}$ and $h(S)$ is (the absolute value of) the
volume (with respect to $M$) of the parallelepiped spanned by $\beta_i\in S$.  
\end{corollary}
\subsection{Group action on zonotopes}
\begin{lemma}\label{lem:fixed_point_zonotope} 
Let $Z=t+\sum_{i\in I} [0,\beta_i]$ be a zonotope in $V$. Let~$G$
be a finite group of affine automorphisms of $V$ which preserves $Z$ and let $G^0$ be the underlying
linear group.
The invariant polytope $Z^G$ is the zonotope
in $V^G$ given by
\begin{equation}
  \label{eq:averaging}
Z^G=\bar{t}+\sum_{i\in I} [0,\hat{\beta}_i]
\end{equation}
where $\bar{?}$, $\hat{?}$ denotes averaging for respectively $G$ and $G^0$; i.e. 
$\bar{u}=(1/|G|)\sum_{\sigma \in G} \sigma(u)$.
$\hat{u}=(1/|G|)\sum_{\sigma \in G^0} \sigma(u)$.
\end{lemma}
\begin{proof}
Since $Z$ is preserved by $G$ and convex it is closed under averaging. Hence 
\[
Z^G=\{\bar{u}\mid u\in Z\}.
\]
We now use the linearity properties of averaging to obtain \eqref{eq:averaging}.
 \end{proof}
 \begin{corollary} \label{cor:lattice}
   Fix a lattice $M$ in $V$ and let $Z$
 be a translation of a lattice zonotope in $V$. Let~$G$
be a finite group of affine automorphisms of $V$ which preserves $Z$.
Then $Z^G$ is the translation of a lattice zonotope.
\end{corollary}
\begin{proof} 
Let $\gamma_i$ be the vectors corresponding to the
  edges of $Z$. By definition $\gamma_i\in M$. 
The
  $\gamma_i$ are only determined up to sign but they yield a
  canonical multiset
  $\{\beta_j\mid j\in J\}:=\{\pm \gamma_i\mid i \in I\}\subset M$.
It follows from Remark \ref{rem:edges} that $Z$ may be written as:
  \[
    Z=t+\sum_{j\in J} [0,\beta_j/2]
  \]
  for suitable $t$. Note that $\{\beta_j\mid j\in J\}$ is preserved by $G^0$ (since the set of edges is preserved by $G$).
  It follows that $t$ is preserved by $G$. 
  By Lemma \ref{lem:fixed_point_zonotope} we have
  \begin{equation}
    \label{lem:canonicalinvariants}
    Z^G=t+\sum_{j\in J} [0,\hat{\beta}_j/2].
  \end{equation}
We consider $G^0$ as acting on $J$.  The right-hand side of \eqref{lem:canonicalinvariants}
can be written as
  \[
    Z^G=t+\sum_{j\in J/G^0} [0,\tilde{\beta}_j/2]
  \]
  where $\tilde{\beta}_j=\sum_{\beta_k\in G^0(\beta_j)} \beta_k\in M$. We have
  $\widetilde{(-\beta_j)}=-\tilde{\beta}_j$ and if $G^0(-\beta_j)=G^0(\beta_j)$
  then $\tilde{\beta}_j=0$. Thus we obtain
  \[
        Z^G=s+\sum_{j\in (J/G^0)/{\pm}} [0,\tilde{\beta}_j]
      \]
      for suitable $s$. Hence $Z^G$ is indeed a translated lattice zonotope.
\end{proof}
\subsection{The M\"obius function of set partitions}
Let $S$ be a finite set. We consider the poset $\Pi_S$ 
of partitions $\Sscr=\{S_1,\ldots,S_t\}$ of $S$, ordered
by refinement (e.g. $\{S',S\setminus S'\}<\{S\}$ for $S'\subset S$). Part of the corresponding M\"obius function \cite[p.7]{Stanleybook}
is given by
\[
\mu(\Sscr):=\mu(\bigl\{\{1\},\ldots,\{n\}\bigr\},\Sscr)=\prod_{S\in \Sscr} \mu(S)
\]
with 
\[
\mu(S)=(-1)^{|S|-1}(|S|-1)!.
\]
Or summarizing
\[
\mu(\Sscr)= (-1)^{n-|\Sscr|} \prod_{S\in \Sscr}
(|S|-1)!.
\]
\subsection{Lattice points in $\Delta^{m,n}_\tau$ and spanning trees}
\label{sec:associating graphs}
In section and the next we assume that~$\tau$ is admissible (see Definition
\ref{def:admissible}).  Let $\Vscr$ be the multiset
$\{(e_i-e_j)^m\mid 1\leq j<i\leq n\}\cup \{e_i\mid 1\leq i\leq
n\}$ (i.e. each $e_i-e_j$, $1\leq j<i\leq n$ in the multiset appears $m$ times). Then $\Delta^{m,n}_\tau$ is a translation of the lattice zonotope
$\sum_{f\in \Vscr}[0,f]$. Hence we may appy Corollary
\ref{thm:stanley} to compute $|\Delta^{m,n}_\tau\cap \ZZ^n|$. Thus we need to
find all subsets of~$n$ linearly independents elements from $\Vscr$.

To this end we create an undirected graph $G$ with vertices $\{0,\ldots,n\}$ with an edge  between $i$ and $j$ for  each vector $e_i-e_j$ in $\Vscr$ 
and an edge between $0$ and $i$ for each vector $e_i$ in $\Vscr$.

Put $[n]=\{1,\ldots,n\}$. So summarizing $G$ is the graph with vertices
$\{0,\ldots,n\}$ and edges
\begin{itemize}
\item $m$ edges between $i$ and $j$ for every $i\neq j\in [n]$,
\item 1 edge between $0$ and $i$ for every $i\in[n]$.
\end{itemize}
 We  obtain the following correspondence that follows by construction and the definition of spanning trees. 
\begin{lemma}\label{lem:lin_ind:sp_trees}
The above correspondence between the elements in $\Vscr$ and edges in $G$ gives a  bijective correspondence between the subsets of $n$ linearly independent elements from $\Vscr$ and the set of spanning trees in $G$. 
\end{lemma}
 For a graph $H$ we denote by $\tau(H)$ the set of its spanning trees.
\begin{corollary}\label{cor:Pspantree}
We have
\[|\tau(G)|=|\Delta^{m,n}_\tau\cap \ZZ^n|.
\]
\end{corollary}

\begin{proof}
  We apply Corollary \ref{thm:stanley}. We may do this since $\partial (\Delta^{m,n}_\tau)\cap \ZZ^n=\emptyset$. Note that
  the volumes of the parallelepipeds (given by suitable minors) are all equal to~$1$. Hence Lemma \ref{lem:lin_ind:sp_trees} implies the 
equality. 
\end{proof}
For every set partition $\Sscr$ of $[n]$ we denote by $G/\Sscr$ 
the  contracted graph, obtained by shrinking all subgraphs connecting vertices in
the same $S\in \Sscr$ to a point
 (note the special role of the vertex $0$, which is never contracted).

\subsection{Regular orbits and spanning trees}
Recall that $S_n$ acts on $\Delta^{m,n}_\tau$. 
We extend Corollary \ref{cor:Pspantree} to $G/\Sscr$.
Denote by~$H_\Sscr$
the stabilizer of $\Sscr$, i.e.
$H_\Sscr=\{g\in S_n\mid gS=S, S\in\Sscr\}$.
We write $(-)^\Sscr$ and $(-)_\Sscr$ for respectively the invariants and the coinvariants under $H_\Sscr$. 
Concretely if
$L_\Sscr=\sum_{i,j\in S\in \Sscr}\ZZ(e_i-e_j)$ then
$\RR^n_{\Sscr}=\RR^n/(L_\Sscr)_\RR$, $\ZZ^n_{\Sscr}=\ZZ^n/L_\Sscr$. We denote
by $({\Delta^{m,n}_\tau})_\Sscr$ the image of ${\Delta^{m,n}_\tau}$ in
$\RR^n_{\Sscr}$.
Finally we put
${}_{\Sscr}(\Delta^{m,n}_\tau)=\{x\in {\Delta^{m,n}_\tau}\mid{\rm
  Stab}(x)=H_\Sscr\}$.
\begin{lemma}\label{lem:invco} We have
  \[
 |(\Delta^{m,n}_\tau)_\Sscr \cap (\ZZ^n)_{\Sscr}|=   \left(\prod_{S\in \Sscr}|S|\right) |(\Delta^{m,n}_\tau)^\Sscr \cap (\ZZ^n)^{\Sscr}|.
    \]
  \end{lemma}
  \begin{proof}
    Since $\partial((\Delta^{m,n}_\tau)^\Sscr)\subset \partial (\Delta^{m,n}_\tau)$ and $\tau$ is admissible we have
    $\partial((\Delta^{m,n}_\tau)^\Sscr) \cap (\ZZ^n)^{\Sscr}=\emptyset$. Likewise using the proof of Lemma \ref{lem:admissible} in \S\ref{sec:admissible} we see that
$\partial (({\Delta^{m,n}_\tau})_\Sscr)\cap \ZZ^n_{\Sscr}=\emptyset$ whenever $\tau-m(n-1)/2$ can be written as a fraction whose denominator is not a sum of
cardinalities of elements of $\Sscr$. Since we have assumed that $\tau$ is admissible, this holds in our case.

We apply  Corollary \ref{thm:stanley} to conclude that $|(\Delta^{m,n}_\tau)_\Sscr \cap (\ZZ^n)_{\Sscr}|=\Vol_{ (\ZZ^n)_{\Sscr}}((\Delta^{m,n}_\tau)_\Sscr)$.
We note that  $(\Delta^{m,n}_\tau)^\Sscr$ is a translate lattice polytope by Corollary \ref{cor:lattice}.
Hence by applying Corollary \ref{thm:stanley} again we obtain 
$|(\Delta^{m,n}_\tau)^\Sscr \cap (\ZZ^n)^{\Sscr}|=\Vol_{ (\ZZ^n)^{\Sscr}}((\Delta^{m,n}_\tau)^\Sscr)$. 

We note that $(\Delta^{m,n}_\tau)^\Sscr\to (\Delta^{m,n}_\tau)_\Sscr$ obtained from $(\RR^n)^\Sscr\to (\RR^n)_\Sscr$ is a bijection, while the cokernel
of $(\ZZ^n)^\Sscr\hookrightarrow (\ZZ^n)_\Sscr$ has order ${\prod_{S\in \Sscr}|S|}$. (Indeed, the cokernel is isomorphically mapped to $\prod_{S\in \Sscr}(\ZZ/|S|\ZZ)$ where the unit  $e_i$ in the $i$-th component maps to 
$\prod_{S\in \Sscr}\delta_{iS}$.) Hence
\begin{multline*}
|(\Delta^{m,n}_\tau)_\Sscr \cap (\ZZ^n)_{\Sscr}|=\Vol_{ (\ZZ^n)_{\Sscr}}((\Delta^{m,n}_\tau)_\Sscr)=\\\left(\prod_{S\in \Sscr}|S|\right) \Vol_{ (\ZZ^n)^{\Sscr}}((\Delta^{m,n}_\tau)^\Sscr)=\left(\prod_{S\in \Sscr}|S|\right)|(\Delta^{m,n}_\tau)^\Sscr \cap (\ZZ^n)^{\Sscr}|.\qedhere
\end{multline*}
\end{proof}

\begin{corollary}\label{cor:Pspantreequotient}
We have
\begin{align*}
  |\tau(G/\Sscr)|=
  \left(\prod_{S\in \Sscr}|S|\right) \sum_{\Sscr'\geq \Sscr}  |{}_{{\Sscr'}}(\Delta^{m,n}_\tau) \cap \ZZ^n|.
\end{align*}
\end{corollary}

\begin{proof}
During the proof of Lemma \ref{lem:invco} we have shown that
$\partial (({\Delta^{m,n}_\tau})_\Sscr)\cap \ZZ^n_{\Sscr}=\emptyset$.
%
It then follows as in Corollary  \ref{cor:Pspantree} that
$
 \tau(G/\Sscr)=|(\Delta^{m,n}_\tau)_\Sscr \cap (\ZZ^n)_{\Sscr}|
 $
 so that we must prove
 \[
|(\Delta^{m,n}_\tau)_\Sscr \cap (\ZZ^n)_{\Sscr}|=
  \left(\prod_{S\in \Sscr}|S|\right) \sum_{\Sscr'\geq \Sscr}  |{}_{{\Sscr'}}(\Delta^{m,n}_\tau) \cap \ZZ^n|.
   \]

This follows from Lemma \ref{lem:invco}.
\end{proof}
If $X$ is a set with an $S_n$-action then we write $\operatorname{reg}(X)$ for the set of regular orbits.
Applying M\"obius inversion we obtain the following formula for $|{\rm reg}({\Delta^{m,n}_\tau}\cap \ZZ^n)|$.
\begin{corollary}\label{cor:Mobius_inversion}
We have
\begin{equation}\label{eq:mobius}
|{\rm reg}({\Delta^{m,n}_\tau}\cap \ZZ^n)|= \frac{1}{n!}\sum_{\Sscr\in \Pi_{[n]}} \frac{1}{\prod_{S\in \Sscr} |S|}\mu(\Sscr) |\tau(G/\Sscr)|.
\end{equation}
\end{corollary}

\begin{proof}
We have  $|{\rm reg}({\Delta^{m,n}_\tau}\cap \ZZ^n)|=(1/n!) |{}_{\Sscr_0}(\Delta^{m,n}_\tau)\cap \ZZ^n|$
with $\Sscr_0=\{\{1\},\ldots,\{n\}\}$.
Now we apply M\"obius inversion to
\[f,g:\Pi_{[n]}\to \QQ,\quad f:\Sscr\mapsto\frac{1}{\prod_{S\in \Sscr}|S|}|\tau(G/\Sscr)|,\quad g:\Sscr\mapsto 
|{}_{\Sscr}(\Delta^{m,n}_\tau)\cap \ZZ^n|
\]
using Corollary \ref{cor:Pspantreequotient}. 
\end{proof}
\subsection{Computing the sum over graphs}
In this section we complete the numeric claim in Corollary \ref{cor:cor1} by computing the sum on right-hand side of \eqref{eq:mobius}.
\begin{theorem} \label{th:mainth}
We have
\begin{equation}
\label{eq:graphsum}
\frac{1}{n!}\sum_{\Sscr\in \Pi_{[n]}} \frac{1}{\prod_{S\in \Sscr} |S|}\mu(\Sscr) |\tau(G/\Sscr)|
= \frac{1}{(m-1)n+1} {mn\choose n}.
\end{equation}
\end{theorem}

\begin{proof}
We use Kirchhoff's theorem which says that the number of spanning trees of a graph is the
absolute value of an arbitrary cofactor in the Laplacian matrix.
The Laplacian matrix of $G$ is the matrix of size $(n+1)\times (n+1)$ given by
\begin{equation}
\label{eq:lap1}
\begin{pmatrix}
-n & 1&1 &\cdots& 1\\
1& -(n-1)m-1 & m &\cdots & m\\
1& m & -(n-1)m-1 &\cdots & m\\
\vdots & \vdots & \vdots && \vdots\\
1&m&m&\cdots &-(n-1)m-1
\end{pmatrix}
\end{equation}
We will look at the cofactor consisting of the rows $[1:n]$ and columns $[2:(n+1)]$. 
By substracting $m$ times the top row $(1,\dots,1)$ from the other rows one finds that it is (up to sign)
the determinant of the square submatrix with rows $[1:n]$  in the following (rectangular) $(n+1)\times n$ matrix.
\begin{equation}
\label{eq:lap2}
\begin{pmatrix}
 1&1 &\cdots& 1\\
 -nm-1 & 0 &\cdots & 0\\
 0 & -nm-1 &\cdots & 0\\
 \vdots & \vdots && \vdots\\
0&0 &\cdots &-nm-1
\end{pmatrix}
\end{equation}
which is (up to sign)
\[
(nm+1)^{n-1}.
\]
Now we evaluate the corresponding cofactor for $G/\Sscr$ where $\Sscr$ is the set partition
\[
\{\{1,\ldots,n_1\},\{n_1+1,\ldots,n_1+n_2\},\ldots\}
\]
(thus the sizes of the parts are $n_1,n_2,\ldots,n_t$).
This amounts to replacing the top row and leftmost column in \eqref{eq:lap1} by 
\[
(-n,n_1,\ldots,n_t).
\]
In the lower left $n\times n$-matrix,
the $n_i\times n_j$ blocks are replaced by their sum. 
We obtain the following matrix
\[
\begin{pmatrix}
-n&n_1&n_2&\cdots &n_t\\
n_1&-n_1((n-n_1)m+1)&n_1n_2m&\cdots&n_1n_tm\\
n_2&n_2n_1m&-n_2((n-n_2)m+1)&\cdots&n_2n_tm\\
\vdots&\vdots&\vdots&&\vdots\\
n_t&n_tn_1m&n_tn_2m&\cdots&-n_t((n-n_t)m+1)
\end{pmatrix}
\]

The matrix corresponding to \eqref{eq:lap2}
becomes
\[
\begin{pmatrix}
 n_1&n_2 &\cdots& n_t\\
 -n_1(nm+1) & 0 &\cdots & 0\\
 0 & -n_2(nm+1) &\cdots & 0\\
 \vdots & \vdots && \vdots\\
 0&0 &\cdots &-n_{t}(nm+1)
\end{pmatrix}
\]
and now we have to calculate the determinant of the top $t$ rows. We find (up to sign)
\[
n_1\cdots n_t (nm+1)^{t-1}.
\]
Hence \eqref{eq:graphsum} becomes
\begin{align*}
\frac{1}{n!}\sum_{\Sscr\in \Pi_{[n]}}& \frac{1}{\prod_{S\in \Sscr} |S|}\mu(\Sscr)
 (mn+1)^{|\Sscr|-1}\prod_{S\in \Sscr} |S|\\
&=
\frac{(-1)^n}{n!}\sum_{\Sscr\in \Pi_{[n]}} (-1)^{|\Sscr|}
 (mn+1)^{|\Sscr|-1}\prod_{S\in \Sscr} (|S|-1)!\\
&=\frac{(-1)^n}{n!}\sum_{n_1+\ldots +n_t=n} (-1)^{t} \frac{1}{t!}\binom{n}{n_1\cdots n_t}
 (mn+1)^{t-1}\prod_{i} (n_i-1)!\\
&=(-1)^n\sum_{n_1+\ldots +n_t=n} (-1)^{t} \frac{1}{t!}
 (mn+1)^{t-1}\prod_{i=1}^t \frac{1}{n_i}
\end{align*}
(the factor $1/t!$ comes from the choice in enumerating the elements of $\Sscr$).
Putting $X=mn+1$ in the following lemma then finishes the proof of Theorem \ref{th:mainth}.
\end{proof}

\begin{lemma} \label{lem:fun} We have
\begin{equation}
\label{eq:identity}
(-1)^n\sum_{n_1+\ldots +n_t=n} (-1)^{t} \frac{1}{t!}
 X^{t-1}\prod_{i=1}^t \frac{1}{n_i}
=\frac{1}{n!}(X-(n-1))\cdots (X-1).
\end{equation}
\end{lemma}
\begin{proof}
To compute the lefthand side we need the coefficient of $u^n$ in the sum
\begin{align*}
\sum_{t\ge 1} \sum_{(n_i)_{i}\in \NN_{>0}^t} (-1)^{t}(-u)^{\sum_{i=1}^t n_i} \frac{1}{t!}
 X^{t-1}\prod_{i=1}^t \frac{1}{n_i}
&= 
\sum_{t\ge 1}\frac{X^{t-1}}{t!} \sum_{(n_i)_{i}\in\NN_{>0}^t} \prod_{i=1}^t \left(-\frac{(-u)^{n_i}}{n_i}\right)\\
&=\sum_{t\ge 1} \frac{X^{t-1}}{t!}\prod_{i=1}^t
\sum_{n\ge 1}\left(-\frac{(-u)^{n}}{n}\right)\\
&=\sum_{t\ge 1} \frac{X^{t-1}}{t!}\log(1+u)^t\\
&=X^{-1} (\exp(X\log(1+u))-1)\\
&=X^{-1} ((1+u)^{X}-1)\\
&=\sum_{n\ge 1} u^n \frac{(X-1)\cdots(X-(n-1))}{n!}.
\end{align*}
Hence the sought coefficient is precisely the righthand side of \eqref{eq:identity}. This finishes the proof.
\end{proof}

\subsection{Second proof of the numeric claim in Corollary \ref{cor:cor1}}\label{subsec:proof2}
The claim that $|{\rm reg}({\Delta^{m,n}_\tau}\cap \ZZ^n)|=A_n(m,1)$ follows by  combining  \eqref{eq:dyckformula}, Corollary \ref{cor:Mobius_inversion} and Theorem \ref{th:mainth}.

\section{Tables of tilting bundles}
\label{sec:tables}
If in \eqref{eq:tilting} we replace $\tau$ by $\tau+u$, $u\in \ZZ$ then
this amounts
to tensoring $\Tau_\tau$ with the line bundle
$\Vscr(u,\ldots,u)$ which does not change its endomorphism ring.
Moreover if $[\tau,\tau']$ does not intersect the inadmissible
locus then $\Tau_\tau=\Tau_{\tau'}$. Hence without loss of generality we may assume
\[
  \tau=\varepsilon+t+m(n-1).
\]
with $0<\varepsilon\ll 1$ and
\[
t=-\frac{p}{k}\qquad 1\le k \le n,\; 0\le p<k.
\]
Below we list the weights $\xi\in (\ZZ^n)^+\cap (\Delta^{m,n}_\tau-\hat{\rho})$ that determine the tilting bundle $\Tscr_\tau$ for varying $t$ and $n=2,3,4$, $m=2$. They were computed with a
computer program.
\subsection{Tilting bundles for \boldmath $(n,m)=(2,2)$}
\[
\begin{array}{lcc}
t=-0/1& [1, 0]&[1, 1]\\
t=-1/2&[0, 0]&[1, 0]
\end{array}
\]
\subsection{Tilting bundles for \boldmath $(n,m)=(3,2)$}
\[
\begin{array}{lccccc}
t=-0/1&
[1, 1, 1]&
[2, 1, 0]&
[2, 1, 1]&
[2, 2, 0]&
[2, 2, 1]\\
t=-1/3&
[1, 1, 0]&
[1, 1, 1]&
[2, 1, 0]&
[2, 1, 1]&
[2, 2, 0]\\
t=-1/2&
[1, 1, 0]&
[2, 0, 0]&
[1, 1, 1]&
[2, 1, 0]&
[2, 1, 1]\\
t=-2/3&
[1, 0, 0]&
[1, 1, 0]&
[2, 0, 0]&
[1, 1, 1]&
[2, 1, 0]
\end{array}
\]
\subsection{Tilting bundles for \boldmath $(n,m)=(4,2)$}
\[
\begin{array}{lcccccccccccccc}
t=-0/1&
[3, 1, 1, 1]&
[2, 2, 1, 1]&
[3, 2, 1, 0]&
[2, 2, 2, 0]&
[3, 2, 1, 1]&
[2, 2, 2, 1]&
[3, 3, 1, 0]\\
&
[3, 2, 2, 0]&
[3, 3, 1, 1]&
[3, 2, 2, 1]&
[2, 2, 2, 2]&
[3, 3, 2, 0]&
[3, 3, 2, 1]&
[3, 2, 2, 2]\\*[1mm]
t=-1/4&
[2, 1, 1, 1]&
[2, 2, 1, 0]&
[3, 1, 1, 1]&
[2, 2, 1, 1]&
[3, 2, 1, 0]&
[2, 2, 2, 0]&
[3, 2, 1, 1]\\
&
[2, 2, 2, 1]&
[3, 3, 1, 0]&
[3, 2, 2, 0]&
[3, 3, 1, 1]&
[3, 2, 2, 1]&
[2, 2, 2, 2]&
[3, 3, 2, 0]\\*[1mm]
t=-1/3&
[2, 1, 1, 1]&
[3, 1, 1, 0]&
[2, 2, 1, 0]&
[3, 1, 1, 1]&
[2, 2, 1, 1]&
[3, 2, 1, 0]&
[2, 2, 2, 0]\\
&
[3, 2, 1, 1]&
[2, 2, 2, 1]&
[3, 3, 1, 0]&
[3, 2, 2, 0]&
[3, 3, 1, 1]&
[3, 2, 2, 1]&
[2, 2, 2, 2]\\*[1mm]
t=-1/2&
[1, 1, 1, 1]&
[2, 1, 1, 0]&
[2, 2, 0, 0]&
[2, 1, 1, 1]&
[3, 1, 1, 0]&
[2, 2, 1, 0]&
[3, 2, 0, 0]\\
&
[3, 1, 1, 1]&
[2, 2, 1, 1]&
[3, 2, 1, 0]&
[2, 2, 2, 0]&
[3, 2, 1, 1]&
[2, 2, 2, 1]&
[3, 2, 2, 0]\\*[1mm]
t=-2/3&
[1, 1, 1, 1]&
[2, 1, 1, 0]&
[3, 1, 0, 0]&
[2, 2, 0, 0]&
[2, 1, 1, 1]&
[3, 1, 1, 0]&
[2, 2, 1, 0]\\
&
[3, 2, 0, 0]&
[3, 1, 1, 1]&
[2, 2, 1, 1]&
[3, 2, 1, 0]&
[2, 2, 2, 0]&
[3, 2, 1, 1]&
[2, 2, 2, 1]\\*[1mm]
t=-3/4&
[1, 1, 1, 0]&
[2, 1, 0, 0]&
[1, 1, 1, 1]&
[2, 1, 1, 0]&
[3, 1, 0, 0]&
[2, 2, 0, 0]&
[2, 1, 1, 1]\\
&
[3, 1, 1, 0]&
[2, 2, 1, 0]&
[3, 2, 0, 0]&
[3, 1, 1, 1]&
[2, 2, 1, 1]&
[3, 2, 1, 0]&
[2, 2, 2, 0]\\
\end{array}
\]


\end{document}